\documentclass [11pt]{article}
\usepackage{amsthm}
\usepackage{amssymb}
\usepackage{epsfig}
\usepackage{amsthm}
\usepackage{srcltx}
\usepackage[dvipsnames,usenames]{color}
\usepackage{paralist}
\usepackage{amsmath}
\usepackage{amsfonts}

\newcounter{contador}
\newcounter{teoA}

\newtheorem{propo}[contador]{Proposition}
\newtheorem{teo}[contador]{Theorem}
\newtheorem{lem}[contador]{Lemma}
\newtheorem{defi}[contador]{Definition}

\theoremstyle{remark}
\newtheorem{nota}[contador]{Remark}

\newcounter{ex}

\newcommand{\R}{{\mathbb R}}
\newcommand{\C}{{\mathbb C}}
\newcommand{\N}{{\mathbb N}}

\newcommand{\U}{{\cal{U}}}

\title{Bifurcation of 2-periodic orbits\\ from non-hyperbolic fixed points
\footnote{The authors are supported by
Ministry of Economy, Industry and Competitiveness of the Spanish
Government through grants MINECO/FEDER MTM2016-77278-P  (first and
second authors) and DPI2016-77407-P (AEI/FEDER, UE, third author). The first  and
second authors are also supported by the grant 2014-SGR-568  from
AGAUR, Generalitat de Catalunya. The third author is supported by
the grant 2014-SGR-859 from AGAUR, Generalitat de Catalunya.}}

\author{Anna Cima$^{(1)}$, Armengol Gasull$^{(1)}$ and V\'{\i}ctor Ma\~{n}osa $^{2}$
  \\*[.1truecm]
{\small \textsl{$^{(1)}$ Departament de Matem\`{a}tiques, Facultat
de Ci\`{e}ncies,}}
\\*[-.25truecm] {\small \textsl{Universitat Aut\`{o}noma de Barcelona,}}
\\*[-.25truecm] {\small \textsl{08193 Bellaterra, Barcelona, Spain}}
\\*[-.25truecm] {\small \textsl{cima@mat.uab.cat,
gasull@mat.uab.cat}}\\
\\*[-.25truecm] {\small \textsl{$^{2}$ Departament de Matem\`{a}tiques,}}
\\*[-.25truecm] {\small \textsl{Universitat Polit\`{e}cnica de Catalunya}}
\\*[-.25truecm] {\small \textsl{Colom 1, 08222 Terrassa, Spain}}
\\*[-.25truecm] {\small \textsl{victor.manosa@upc.edu}}}
\begin{document}

\maketitle
\begin{abstract}
We introduce the concept of $2$-cyclicity for families of
one-dimensional maps with a non-hyperbolic fixed point by analogy
to the  cyclicity for families of planar vector fields with a weak
focus. This new concept is useful in order to study the number of
$2$-periodic orbits that can bifurcate from the fixed point. As an
application we study  the $2$-cyclicity of some natural families of
polynomial maps.
\end{abstract}




\section{Introduction}\label{S-intro}

The cyclicity of a family of vector fields having a  weak focus or a
center is a well known concept in the theory of planar vector fields
and the problems surrounding the second part of the Hilbert's 16th
problem  \cite{Il,R}. A grosso modo the cyclicity expresses the
maximum number of small amplitude limit cycles that can effectively
bifurcate from the singular point by varying the parameters in the
family of considered vector fields.

This cyclicity  is given by the number of fixed points near the
critical point of a family of orientation preserving maps (the so
called return maps)  with a non-hyperbolic fixed point. As we will
see, the cyclicity also can be seen as the number of 2-periodic
orbits of a related family of orientation reversing maps (the
half-return maps), see for instance \cite{CGP} or Section
\ref{SS-bifurcation}. Recall that given a map $f:\R\to\R$, a
2-periodic orbit is a set $\{x,y\}$ such that $f(x)=y, f(y)=x$ and
$x\ne y.$

Hence it is natural, in the discrete setting, to study the
bifurcation of  $2$-periodic orbits from non-hyperbolic fixed points
of orientation reversing one-dimensional  analytic diffeomorphisms
of the form
\begin{equation}\label{E-maps}
f(x)=f_a(x)=-x+\sum\limits_{j\geq 2} a_j x^j.
\end{equation}
This will be the main goal of this paper.

To fix the problem  we start introducing the concept of
\emph{$2$-cyclicity} of a family of maps of the form \eqref{E-maps},
by analogy with the concept of cyclicity for planar vector fields.
Here, given $x\in\R^m$ and $\rho\in\R^+,$
$D_\rho(x):=\{y\in\R^m\,:\, ||y-x||<\rho\}.$

\begin{defi}\label{D-cycl}
Set $a=(a_{1},\ldots,a_{n})$ varying in an open set of
$\mathcal{V}\subseteq\R^{n}$, and consider the family of analytic reversing
orientation maps from $\R$ into itself,
\begin{equation}\label{E-mapsaestrella}
f_{a}(x)=-x+\sum_{i\geq 2} c_{i}(a) x^i.
\end{equation} We will say that the origin of a map $f_{a^*}$, with  $a^*\in\mathcal{V}$,
has $2$\emph{-cyclicity} $N\in\mathbb{N}\cup\{0\}$ if:
\begin{enumerate}
  \item[(i)] it is possible to find $\varepsilon_0>0$ and
  $\delta_0>0$ such that the maximum of  isolated
   $2$-periodic orbits within $D_{\delta_0}(0)\subset
  \R$ for every map \eqref{E-mapsaestrella} with $a\in D_{\varepsilon_0}(a^*)\subset \mathcal{V}$ is
$N.$

  \item[(ii)] for any $\varepsilon>0$ and any  $\delta>0$
  there exists $a\in D_{\varepsilon}(a^*)\subset\mathcal{V}$
  such that $f_a$ has $N$ different  isolated $2$-periodic orbits within
  $D_{\delta}(0).$
\end{enumerate}
A family of maps $f_a$,  with $a\in\mathcal{V}\subseteq\R^{n}$, has
$2$-cyclicity $N$ at the origin if $N$ is the maximum $2$-cyclicity
achieved by a map in the family.
\end{defi}

We remark that it has no sense to study the 2-cyclicity for locally
orientation preserving diffeomorphisms because it is always 0, see
Remark \ref{P-cic0op}.

 In the
following, for the sake of simplicity,  we will simply say
\emph{cyclicity} to refer to {\it $2$-cyclicity} of a map or a
family of maps at the origin.

\medskip

We remark that the cyclicity of a family \eqref{E-mapsaestrella}
does not depend only on the number $n$ of parameters but on their
role. As an example, we will show in Section \ref{SS-bifurcation}
that there exist one-parametric families ($n=1$) of maps  with
arbitrary large cyclicity.

\medskip

 In the recent paper \cite{CGM17}, we have introduced what we
 call \emph{stability constants}  to study the stability of the origin of one-dimensional
 maps of the form
 \eqref{E-mapsaestrella} and also of  periodic discrete dynamical
  systems with a common fixed point. A summary of results on this issue can also be
found in \cite{DEP}. The analysis of these constants plays also an
important role in the study of the cyclicity, as the proof of our
main result of this paper evidences. Let us recall them.

To know  the local stability of the origin of an analytic map of the
form \eqref{E-maps}, we consider
\begin{equation}\label{E-Ws}
f\circ f(x):=f(f(x))=x+\sum\limits_{j\geq 3} W_j(a_2,\ldots,a_j)
x^j.
\end{equation}

If $f$ is not an involution (i.e. $f\circ f\neq\mathrm{Id}$), we
define a \emph{stability constant of order} $k$ (with $k\geq 3$) as
\begin{equation}\label{E-stab-cons}
    \begin{array}{l}
      V_3=V_3(a)=W_3(a_2,a_3) \,\hbox{ if }\, W_3\neq 0,\, \hbox{ or}\\
      V_k=V_k(a)=W_k(a_2,\ldots,a_k) \,\hbox{ if }\, W_j=0,\,j=3,\ldots,k-1.
    \end{array}
\end{equation}
Notice that the stability constant $V_k$ only has sense when all the
previous $W_j, j<k$ vanish. Hence, any expression of the form
$W_k+U,$ where $U$ belongs to the ideal generated by
$W_3,W_4,\ldots, W_{k-1},$  $\mathcal{I}_{k-1}:=\langle
W_3,W_4,\ldots, W_{k-1}\rangle$, is a valid expression for $V_k.$ In
this work we will refer the expressions of the polynomials $W_k$ as
stability constants,
 but also we will consider the expressions $V_k$ as the normal forms of $W_k$ in the Gr\"{o}bner basis
of $\mathcal{I}_k$ when the graded reverse lexicographic order
(called \emph{grevlex} or \emph{degrevlex} in the literature and
$\operatorname{tdeg}(a_2,a_3,\ldots,a_k)$ in Maple) is used, see
\cite[p.~58]{CLS}. In this order, the monomials are compared first
by their total degree and ties are broken by reverse
lexicographic order, that is, by smallest degree in
$a_{k},a_{k-1},\ldots,a_2$. In order to avoid ambiguity, the
expressions of $V_k$ will be called \emph{reduced} stability
constants.

It is known that the first non-zero stability constant is of odd
order (see \cite{CGM17}). For the sake of completeness, in Section
\ref{S-estabilitat} we  include a proof of this fact, as well as
their algebraic properties that are reminiscent of similar
properties satisfied by the Lyapunov and period constants, see
\cite{ChJ,CGM97,GT,LT,Z}.

If for a value of $a$ it holds that
$V_3(a)=V_5(a)=\cdots=V_{2k-1}(a)=0$ and $V_{2k+1}(a)\ne0$ we will
say that the origin is a {\it weak fixed point of order $k-1,$} by
similitude with the concept of order of a weak focus for
non-degenerated critical point of planar polynomial vector fields.
As we will see in Proposition \ref{p:lower} the maximum number of
2-periodic orbits that bifurcate from a weak fixed point of order
$m$ is $m.$

Our main result, which is proved in section \ref{S-prova}, deals
with the simplest case: the maps $f_a$ are polynomial of fixed
degree $d$, and the parameters are the coefficients of the system.
Notice that the only involution in these families corresponds to the
trivial case $f_{\mathbf{0}}(x)=-x.$ As we will see, even in this
simple setting some questions are not easy to answer.

\begin{teo}\label{T-cicl} Consider the  family of polynomial maps
\begin{equation}\label{E-poly-fam-a}
    f_a(x)=-x+\sum\limits_{j=2}^d a_j x^j, \quad a=(a_2,a_3,\ldots,a_d)\in\R^{d-1}.
\end{equation}
 It has only the trivial involution corresponding to
$a={\bf 0}$ and its cyclicity is at most $\left[(d^2-1)/2\right]$,
where $[\,\cdot\,]$ stands for the integer part.  Furthermore:
\begin{enumerate}
\item[(a)] For $d$ even, its cyclicity is at least $d-2.$ Moreover,
\begin{enumerate}[(i)]
\item For $d=2,4$ it is  $d-2.$
\item For $d=6,8,10$ it is at most  $d-2$ for any $f_a,$ $a\neq
\mathbf{0}.$
\item For $d=6,8,10$ it is at most $5,9,13,$ respectively, for $f_{\mathbf{0}}$.
\end{enumerate}

\item[(b)] For $d$ odd, its cyclicity is at least $d-3.$ Moreover,
\begin{enumerate}[(i)]
\item For $d=3$ it is  $d-2=1.$
\item  For $d=5,7,9$ it is at most $d-2$ for any $f_a,$ $a\neq
\mathbf{0},$ and there is some $a$ such that it is $d-2.$
\item For $d=5,7,9$ it is at most $4,7,10,$ respectively, for $f_{\mathbf{0}}$.
\item  For $d=4m+3,\, m\ge0,$ there are some values of $a$ such that
 the origin is a weak fixed point of order $d-2$ for the corresponding $f_a$.
\end{enumerate}
\end{enumerate}
\end{teo}
Observe that the above result only accounts for the number of local
(near $x=0$)  isolated $2$-periodic orbits. For instance, with
respect statement $(a)$ with $d=4$, and although the cyclicity of
the family is $2$, it is easy to find examples with 3 global
$2$-periodic orbits. This is the case, for instance, for the map
$f(x)=-x-7x^2+10x^4,$ which has also four fixed points.
 Notice that, the first statements of the above result
are straightforward. If for some $a,$ $f_a$ has degree $k$ then
$f_a\circ f_a$ has degree $k^2.$ Hence the only involution is
$f_{\mathbf{0}}(x)=-x.$ Moreover,  a priori, the maximum number of
isolated fixed points of $f_a\circ f_a$ for any polynomial map of
degree $d$ is $d^2$. Hence excluding the fixed point $x=0$, we have
that the maxim number of global 2-periodic orbits is
$\left[(d^2-1)/2\right]$. It is not difficult to construct examples
of polynomial maps of degree $d$ (for instance using Chebyshev
polynomials) with $\left[(d^2-d)/{2}\right]$ global 2-periodic
orbits.

\medskip

 It seems natural
 to think that for any $d$ the cyclicity is $d-2.$ For $d$ even,
 we have been able to prove that this value is a lower bound of this cyclicity by using the algebraic
 properties of the stability constants. When  $d$ is odd
 the problem is more difficult. In particular, for $d=4m+5$, it is not easy at all to prove the
 existence of weak fixed points of order $d-2$, see   our proofs
  for
  cases $d=5,9$ in  item (ii) of part (b) of the theorem.

  To  prove that $d-2$ is an upper bound for values of $a$ for which the origin is a weak
 fixed point is sometimes possible because we can use again some algebraic computations
 together with the Weierstrass
 preparation theorem, see Proposition \ref{p:lower} and Lemma \ref{l:rad}. Nevertheless,
 when $a=\mathbf{0},$  our approach needs to show that
 the ideal generated by the first $d-1$ stability constants, say $\mathcal{I},$
  is \emph{radical} and contains all the functions $W_j(a)$ given in
  \eqref{E-stab-cons}. This is only true for $d=2,3,4.$

For $d\geq 4$, the proof of statements (iii) of parts (a) and (b) of Theorem \ref{T-cicl} are based on large symbolic computations.

\medskip

In  Section \ref{SS-bifurcation} we study the relation between the
cyclicity  of weak foci or centers of  planar vector fields and our
results. In particular we show that any map of type \eqref{E-maps}
is a model for the half-return map associated to a weak focus, see
Proposition \ref{P-vf}.

\section{Stability constants and preliminary results}\label{S-estabilitat}

In this section, first we prove some properties of the stability
constants, including also their algebraic properties. Secondly, we
include some standard tools to prove upper or lower bounds for the
cyclicity of families of maps.

A related result to next theorem  is also given in \cite{DEP}, first
in terms of the derivatives of the map $f\circ f$ (Theorem 5.1), and
also using some explicit expressions that are closely related with
the stability constants (Theorem 5.4), which are obtained using
 the Fa\`a di Bruno Formula, \cite{Johnson}.

\begin{teo}\label{T-orient-rev}
Let $f_a$ be an analytic map of the form \eqref{E-maps}. If $f$ is
not an involution, then there exists $m\geq 1$ such that
$V_3=V_5=\cdots=V_{2m-1}=0,$ $V_{2m+1}\neq 0.$ Moreover, if
$V_{2m+1}< 0$ (resp. $>0$), the origin is locally asymptotically
stable (resp. a repeller). In particular, all $V_{2k+1}=V_{2k+1}(a),
k\ge1,$ are polynomials in the variables $a_2,a_3,\ldots,a_{2k+1}$
 and the first reduced stability constants are:
\begin{align*}
  V_3=& -2{a_{{2}}}^{2}-2a_{{3}},\\
  V_5=& -6\,a_{{4}}a_{{2}}+4\,{a_{{3}}}^{2}-2\,a_{{5}},\\
  V_7=& 3\,a_{{2}}a_{{3}}a_{{4}}-8\,a_{{6}}a_{{2}}+13\,a_{{3}}a_{{5}}-4\,{a_{{
4}}}^{2}-2\,a_{{7}},\\
  V_9=& {\frac {242}{17}\,a_{{2}}a_{{3}}a_{{6}}}-{\frac {121}{17}\,a_{{2}}a_{{4}}a_
{{5}}}-10\,a_{{8}}a_{{2}}+{\frac {358}{17}\,a_{{3}}a_{{7}}}-10\,a
_{{4}}a_{{6}}+{\frac
{69}{17}\,{a_{{5}}}^{2}}-2\,a_{{9}},\\
  V_{11}=& {\frac {4563}{121}\,a_{{2}}a_{{3}}a_{{8}}}-{\frac {11765}{242}\,a_{{2}}a_{{4
}}a_{{7}}}+\frac{13}2\,a_{{2}}a_{{5}}a_{{6}}+{\frac
{4407}{242}\,a_{{3}}a_{{ 4}}a_{{6}}}-{\frac
{936}{121}\,a_{{3}}{a_{{5}}}^{2}}\\&-12\,a_{{10} }a_{{2}}+{\frac
{3865}{121}\,a_{{3}}a_{{9}}}-12\,a_{{4}}a_{{8}}+{ \frac
{515}{242}\,a_{{5}}a_{{7}}}-6\,{a_{{6}}}^{2}-2\,a_{{11}},\\
 \end{align*}
  \begin{align*}
  V_{13}=&{\frac {94587200}{1428271}\,a_{{2}}a_{{3}}a_{{10}}}-{\frac {304305945}{2856542}
\,a_{{2}}a_{{4}}a_{{9}}}+{\frac
{2992379}{219734}\,a_{{2}}a_{{5}}a_{{ 8}}}+{\frac
{1939207}{329601}\,a_{{2}}a_{{6}}a_{{7}}}\\&+{\frac
{145516929}{2856542}\,a_{{3}}a_{{4}}a_{{8}}}-{\frac
{138885638}{4284813}\,a_{{3} }a_{{5}}a_{{7}}}+{\frac {4183988}{
1428271}\,a_{{3}}{a_{{6}}}^{2}}-{\frac
{273943}{329601}\,{a_{{4}}}^{2}a_{{7}}}\\&+{\frac {
383791}{109867}\,a_{{4}}a_{{5}}a_{{6}}}-14\,a_{{12}}a_{{2}}+{\frac {
62421386}{1428271}\,a_{{3}}a_{{11}}}-14\,a_{{4}}a_{{10}}-{\frac {
29912981}{2856542}\,a_{{5}}a_{{9}}}\\&-14\,a_{{6}}a_{{8}}+{\frac
{3323839}{329601} \,{a_{{7}}}^{2}}-2\,a_{{13}},\\
  V_{15}=& -{\frac {6188200}{465637}\,a_{{2}}a_{{3}}a_{{5}}a_{{8}}}+{\frac {
964610838}{8847103}\,a_{{2}}a_{{3}}a_{{12}}}-{\frac
{1932055066}{8847103}\,a_{{2 }}a_{{4}}a_{{11}}}\\&+{\frac
{2073461406}{115012339}\,a_{{2}}a_{{5}}a_{{10} }}+{\frac
{102777002}{1396911}\,a_{{2}}a_{{6}}a_{{9}}}+{ \frac
{10885500630}{1070499463}\,a_{{2}}a_{{7}}a_{{8}}}\\&+{\frac {
1324158696}{8847103}\,a_{{3}}a_{{4}}a_{{10}}}-{\frac
{70657783876}{345037017}\,a_{ {3}}a_{{5}}a_{{9}}}+{\frac
{178495020}{8847103}\,a_{{3}}a_{{6}}a_{{8 }}}\\&-{\frac
{10948144126}{1070499463}\,a_{{3}}{a_{{7}}}^{2}}+ {\frac
{888498472}{26541309}\,{a_{{4}}}^{2}a_{{9}}}+{\frac
{2562962080}{115012339} \,a_{{4}}a_{{5}}a_{{8}}}\\&-{\frac
{4032962292}{1070499463}\,a_{{4}}a_{{6} }a_{{7}}}-{\frac
{150876019048}{ 13916493019}\,{a_{{5}}}^{2}a_{{7}}}+{\frac
{546329272}{115012339}\,a_{{5}}{a_{{6}}}^{2}}\\&-16
\,a_{{14}}a_{{2}}+{\frac
{511907618}{8847103}\,a_{{3}}a_{{13}}}-16\,a_ {{4}}a_{{12}}-{\frac
{4393292988}{115012339}\,a_{{5}}a_{{11}}}\\&-16\,a_{
{6}}a_{{10}}+{\frac
{6893660012}{169026231}\,a_{{7}}a_{{9}}}-8\,{a_{{8
}}}^{2}-2\,a_{{15}}.
\end{align*}

\end{teo}

\begin{proof}   First, observe that by the definition of normal
form of $W_k$ using a Gr\"obner basis $\mathcal{G}$ of the ideal
$\langle W_3,W_4,\ldots,W_{k-1}\rangle$ it holds that
$V_k=W_k+\sum_{g\in \mathcal{G}} p_g\, g$  where $p_g$ are
polynomials in $a$, see \cite[p.~82]{CLS}. Hence
$\mathrm{sign}(V_k)=\mathrm{sign}(W_k)$.

 Next we prove that that the order of the first
non-zero stability constant is odd. Suppose, to arrive to a
contradiction, that $f(
f(x))-x=W_{2m}x^{2m}+O(x^{2m+1})=V_{2m}x^{2m}+O(x^{2m+1})$ with
$V_{2m}\ne 0.$ For instance assume that $V_{2m}>0.$  Then we can
consider a neighborhood of the origin $\U$  such that for all $x\in
U\setminus\{0\},$ $f$ is strictly monotonically decreasing and
$f(f(x))-x>0.$ Let $x_0\in U\setminus\{0\}$ and consider its orbit
$x_n=f(x_{n-1}), n\ge1.$ We also take $|x_0|$ small enough, such
that $x_1,x_2,x_3 \in U.$ We know that $x_2-x_0=f( f(x_0))-x_0>0.$
Since $f$ is decreasing, it implies that $f(x_2)<f(x_0),$ that is,
$f(f(x_1))<x_1,$ a contradiction with $f(f(x))-x>0.$

 A simple argument gives that the stability of the origin  for $f\circ f$
is determined by the sign of $x\big(f(f(x))-x\big)$ in a
neighborhood of the origin. Observe that when $V_{2m+1}\ne0$, it
holds that for $x\in U\setminus\{0\}$ the function
$x\big(f(f(x))-x\big)=V_{2m+1}x^{2m+2}+O(x^{2m+3})$ has the same
sign that $V_{2m+1}.$ As a consequence, the stability of the origin
 for  both maps $f\circ f$ and $f$ is characterized by the sign of the stability constants.~\end{proof}

\medskip

We continue this section by proving an algebraic property of the
stability constants $W_k$. This property is analogous to the one
possessed by the Lyapunov constants of  weak foci and the period
constants of centers for planar vector fields, see \cite{CGM97}. In
fact, these constants play a similar role to the \emph{Lyapunov
constants} in the study of small amplitude limit cycles of planar
analytic differential systems with weak focus or a center, or the
the \emph{Period constants} in the study of the critical periods
arising in planar centers, \cite{ChJ,CGM97,GT,LT,Z}. Ending with
this list of similarities, we can say that the case where $f\circ
f=\operatorname{Id}$ is the one corresponding with either the center
or the isochronous cases,  in each of the above two analogous
situations.

\begin{propo}\label{L-quasi-homo} The stability constants $W_j$, introduced in \eqref{E-Ws}, associated
to an orientation reversing diffeomorphism of the form
\eqref{E-maps} are quasi-homogeneous polynomials of quasi-degree
$j-1$ and weights $(1,2,\ldots,j-1)$ in the  coefficients
$(a_2,a_3,\ldots,a_j)$, that is
 $$
    W_j(\lambda a_2,\lambda^2
    a_3,\ldots,\lambda^{j-1}a_{j})=\lambda^{j-1}W_j(a_1,\ldots,a_{j}).
$$
\end{propo}
\begin{proof} It can be seen straightforwardly that each coefficient $W_j$ is a
polynomial function of the coefficients of $a_i$ for $i=2,\ldots,j$.

Observe that the change of variables  $x=\lambda u$ conjugates the
map $f(x)=-x+\sum_{i\geq 2} a_i x^i$ with the map $
g(u)=-u+\sum_{i\geq 2} b_i u^i$ where $b_i=\lambda^{i-1}a_i$, and so
conjugates the map $f(f(x))=x+\sum_{j\geq 3} W_j(a_2,\ldots,a_{j})
x^j,$ with
$$
g(g(u))=u+\sum\limits_{j\geq 3} W_j(b_2,\ldots,b_{j}) u^j.
$$
Since $g(g(u))=\frac{1}{\lambda}f(f(\lambda u))$ we have
\begin{align*}
u+\sum\limits_{j\geq 3} W_j(b_2,\ldots,b_{j}) u^j&=
\frac{1}{\lambda}\left(\lambda u+\sum\limits_{j\geq 3}
W_j(a_2,\ldots,a_{j}) (\lambda u)^j\right)\\&= u+\sum\limits_{j\geq
3} \lambda^{j-1}W_j(a_2,\ldots,a_{j}) u^j.
\end{align*}
Hence $W_j(b_2,\ldots,b_{j})=\lambda^{j-1}W_j(a_2,\ldots,a_{j})$.~
\end{proof}

As we will see, the above result is a key tool to prove part of item
$(a)$ of Theorem \ref{T-cicl}. It is also useful to find algebraic
relations among the polynomials $W_j$ because a priori they give
some restrictions on them.

 As we have already explained in their definition, the
explicit expressions of the reduced stability constants have been
obtained first by computing coefficients of the Taylor expansion of
$f\circ f$ and afterwards, by taking the normal form of $W_k$ in the
Gr\"{o}bner basis of $\langle W_3,W_k,\ldots, W_{k-1}\rangle$ when
the graded reverse lexicographic order  is used. The above results
states that the stability constants $W_k$ are quasi-homogeneous
polynomials in the coefficients of the maps.  Notice that the
reduced stability constants  $V_k,$  given in Theorem
\ref{T-orient-rev}, are also quasi-homogeneous polynomials.

 Next results collect and adapt some  tools for
studying the number of zeroes of families of smooth maps that are
borrowed from the techniques used to study the number of small
amplitude limit cycles bifurcating from weak foci or centers.

\begin{propo}\label{p:upper} Let $W_j=W_j(a)$ and $V_j=V_j(a)$ be the polynomials
associated to the family of maps \eqref{E-poly-fam-a}  given in
\eqref{E-stab-cons}. Assume that there exists $m=m(d)$ such that for
all $k=1,2,\ldots,m-1$,
$$\langle W_3,W_4,\ldots,W_{2k+1}\rangle=\langle
W_3,W_4,\ldots,W_{2k+2}\rangle=\langle
V_3,V_5,\ldots,V_{2k+1}\rangle$$ and $\langle
V_3,V_5,\ldots,V_{2m+1}\rangle=\langle W_3,W_4,\ldots
W_{d^2}\rangle.$ Then the cyclicity of the family is at most $m-1.$
\end{propo}

\begin{proof}
We need to study the number of  isolated positive zeroes in a
neighborhood of the origin of the maps
\begin{equation}\label{e:forma}
h_a(x)=\frac{f_a(f_a(x))-x}{x^3}=\sum_{j= 3}^{d^2}
W_j(a)x^{j-3}=\sum_{j=1}^m
V_{2j+1}(a)\big(1+x\psi_{2j+1}(x,a)\big)x^{2j-2},
\end{equation}
where, to write the last equality, we have used the hypotheses on
the polynomials $W_j$ and $V_{2j+1}$ and $\psi_{2j+1}$ are
polynomial functions. Notice that these zeroes always correspond to
2-periodic orbits of $f_a$ and are not fixed points because,
locally, $f_a$ sends positive values of $x$ to negative ones, and
viceversa.

The procedure that we follow is rather standard and it is usually
called \emph{division-derivation algorithm}. Other examples of its
application can be seen in \cite{ChJ, GLT, R,Z}.

  We will prove by induction that any map of the form
\begin{equation}\label{e:formagen}
h_a(x)=\sum_{j=1}^k
g_{j}(a)\big(1+x\psi_{j}(x,a)\big)x^{2j-2},
\end{equation}
where $\psi_{j}$ are smooth functions in $x$, has at most
$k-1$ positive isolated zeroes is any small enough neighborhood of the
origin.

When $k=1,$ then obviously the function \eqref{e:formagen} has not zeroes.
Assume that the result holds for $k=m-1$. Set $k=m$,
then
\begin{align*}
\frac{h_a(x)}{1+x\psi_{1}(x,a)}=&\sum_{j=1}^mg_j(a)\frac{1+x\psi_{j}(x,a)}{1+x\psi_{1}(x,a)}x^{2j-2}\\=&
g_1(a)+\sum_{j=2}^m
g_{j}(a)\big(1+x\phi_{j}(x,a)\big)x^{2j-2},
\end{align*}
where $\phi_{j}$ are smooth functions in $x$. Then, for some new smooth
functions $\varphi_{j}$ and $\zeta_j$:
\begin{align*}
\frac{d}{dx}\Big(\frac{h_a(x)}{1+x\psi_{1}(x,a)}\Big)=& \sum_{j=2}^m
g_j(a)\big(2j-2+x\varphi_{j}(x,a)\big)x^{2j-3}\\&=\sum_{j=2}^m
(2j-2)\,g_j(a)\big(1+x\zeta_{j}(x,a)\big)x^{2j-3}.
\end{align*}
Observe that the map
\begin{align*}
\frac{1}{x}\frac{d}{dx}\Big(\frac{h_a(x)}{1+x\psi_{1}(x,a)}\Big)=&
\sum_{j=2}^m (2j-2)\,g_j(a)\big(1+x\zeta_{j}(x,a)\big)x^{2j-4}\\=&
\sum_{i=1}^{m-1} \tilde{g}_i(a)\big(1+x\xi_{i}(x,a)\big)x^{2i-2},
\end{align*}
where $ \tilde{g}_i(a)=2i \,g_{i+1}(a)$ and  $\xi_i(x,a)=\zeta_{i+1}(x,a)$, is of the form
\eqref{e:formagen} with $k=m-1$. Hence, by the induction hypothesis it has at most $m-2$
zeroes in any positive
neighborhood of the origin. Hence, by the Rolle's Theorem the map $h_a$ has at most $m-1$ zeroes.

Of course, since the map \eqref{e:forma} is in the form
\eqref{e:formagen}, the result follows. Observe that if for some
values of $a,$ one of the $V_{2j+1}$ vanishes, the division
derivation procedure for this value of $a$ can be accelerated and
gives rise to less number of positive zeroes.
\end{proof}

\begin{propo}\label{p:lower} Let $V_j=V_j(a)$ be the reduced stability constants
associated to the family of maps \eqref{E-poly-fam-a}  given in
\eqref{E-stab-cons}. Assume that for $a=a^*$ the map has a weak
fixed point of order $m-1,$ that is,
$V_3(a^*)=V_5(a^*)=\cdots=V_{2m-1}(a^*)=0$ and $V_{2m+1}(a^*)\ne0.$
Then, the maximum cyclicity of $f_{a^*}$ is $m-1.$

Moreover, if the $m-1$ vectors
\[
\nabla V_3(a^*), \nabla V_5(a^*),\ldots, \nabla V_{2m-1}(a^*),
\]
where $\nabla=(\partial/\partial a_2, \partial/\partial a_3,\ldots,
\partial/\partial a_{m}),$ are linearly independent,  the
cyclicity of the map $f_{a^*}$ is  $m-1.$
\end{propo}

\begin{proof} To prove that the maximum cyclicity of the origin of $f_{a^*}$
is $m-1,$ as usual, we will apply the Weierstrass preparation
theorem (\cite{gr}) to the function $h_a(x)$ introduced in
\eqref{e:forma}. More precisely, write
$H(x,a)=H(x,a_2,a_3,\ldots,a_d)=h_a(x)$ as a
 holomorphic function with $d$-variables. Notice that
\[
H(x,a_2^*,a_3^*,\ldots,a_d^*)=V_{2m+1}(a^*)x^{2m-2}+O(x^{2m-1})
\]
and hence, we are under the hypotheses of that theorem. Therefore,
in a neighborhood in $\C^d$ of $(0,a^*),$ it holds that

\begin{multline}
H(x,a_2,a_3,\ldots,a_d)\\=\big[x^{2m-2}+A_{2m-3}(a)x^{2m-3}+A_{2m-3}(a)x^{2m-3}+\cdots+A_1(a)x+A_0(a)
\big]g(x,a),
\end{multline} where  $A_j$ and $g$ are holomorphic functions,
$g(0,a^*)=V_{2m+1}(a^*)\ne0$ and $A_j(a^*)=0.$ As a consequence, for
parameters in a neighborhood of $a=a^*$ the function $h_a$ has at
most $2m-2$ non-zero roots in a neighborhood of the origin. Since
the non-zero roots of this function appear in couples (for each
positive zero corresponding to a 2-periodic orbit, there is a
negative one corresponding to the other point of this orbit), we
have proved that the number of positive zeroes in a neighborhood of
the origin is at most $m-1,$ giving the desired bound for the
cyclicity.

The proof of the second part is also based on a well-known approach,
see for instance \cite{CGP}. It simply uses Bolzano's theorem and
consists on producing successive changes of stability of the origin.
We give the details when $m=3.$ The general case follows by using
the same type of arguments. Recall that $h_a(x)=(f_a(f_a(x))-x)/x^3$
and its positive zeroes give rise to the 2-periodic orbits.

We have that for $a=a^*$ it holds that $V_3(a^*)=V_5(a^*)=0$ and
$V_7(a^*)\ne0$. Assume without loss of generality that $V_7(a^*)<0.$
If $\delta_2$ is small enough then for all $0<\delta<\delta_2$ there
exists $x_0>0$ such that $|x_0|<\delta$ such that $h_{a^*}(x_0)<0$.
Consider the mapping $\Phi$ from $\R^2$ to $\R^2$ given by
$\Phi(a_2,a_3)= \left(V_3(a_2,a_3,a_4^*),V_5(a_2,a_3,a_4^*)\right).$
Then $\Phi(a_2^*,a_3^*)=(0,0)$ and since by hypothesis ${\nabla}
V_3(a^*),{\nabla V_5}(a^*) $ are linearly independent, $\det\big(
D\Phi(a^*)\big)\ne0.$ This fact implies that $\Phi$ is locally
exhaustive. Hence, we can find values $a=(a_2,a_3,a_4^*)$ as near as
we want of $a^*,$ say $|a-a^*|<\epsilon_1$, with $V_5(a)>0$ and
$V_3(a)=0.$ This fact implies that there exists $0<x_1<x_0<\delta$
such that $h_a(x_1)>0$ but still $h_a(x_0)<0.$ Hence, there exists
at least a positive root of $h_a$ in $(x_1,x_0)$. Now let $a$ with
$|a-a^*|<\epsilon_2<\epsilon_1$ such that $V_3(a)<0$ and, yet
$h_a(x_2)h_a(x_1)<0.$  Finally, we get that there exists $0<x_2<x_1$
satisfying $h_a(x_2)<0$ with $h_a(x_1)>0$ and $h_a(x_0)<0.$ Hence,
for $(x,a)\in(0,\delta)\times D_{\epsilon_2}(a^*),$ $h_a(x)$ has two
positive zeros corresponding with the two announced  2-periodic
orbits.
\end{proof}

 Next proposition extends the second part of the
previous one, when instead of dealing with the reduced stability
constants we consider the stability constants. Its proof is similar
and we omit it.

\begin{propo}\label{p:lower2} Let $W_j=W_j(a)$ be the  stability constants
associated to the family of maps~\eqref{E-poly-fam-a}  given
in~\eqref{E-stab-cons}. Assume that for $a=a^*$ the map has a weak
fixed point of order $m-1,$ that is,
$W_3(a^*)=W_4(a^*)=\cdots=W_{2m-1}(a^*)=W_{2m}(a^*)=0$ and
$W_{2m+1}(a^*)\ne0.$ Then, if the $m-1$ vectors
\[
\nabla W_3(a^*), \nabla W_5(a^*),\ldots, \nabla W_{2m-1}(a^*),
\]
where $\nabla=(\partial/\partial a_2, \partial/\partial a_3,\ldots,
\partial/\partial a_{m}),$ are linearly independent,  the
cyclicity of the map $f_{a^*}$ is at least $m-1.$
\end{propo}

\begin{lem}\label{l:rad}  Let $W_j=W_j(a)$ and $V_j=V_j(a)$ be the stability constants
associated to the family of maps \eqref{E-poly-fam-a} given in
\eqref{E-stab-cons}. Assume that there exist
 $k\geq 3$ and
$0<n_j\in\N$ such that for all $j=3,4,\ldots,d^2$,
\begin{equation}\label{e:Wjnj}
 W_j^{n_j}\in \langle
V_3,V_5,\ldots,V_{2k+1}\rangle.\end{equation}  Let $\ell$ denote the
minimum $k$ such that \eqref{e:Wjnj} holds. Then, the highest order
of the origin as  weak fixed point  is $\ell-1.$ Moreover, the
maximum cyclicity of any map $f_a,$ with $a\ne\mathbf{0},$ is also
$\ell-1$.
\end{lem}

\begin{proof} Assume, to arrive to a contradiction,  that the family has some weak fixed point with
order bigger than $\ell-1$ for some $a=a^*\neq \mathbf{0}.$ In particular, for
this  $a$ it holds that $V_3(a^*)=V_5(a^*)=\cdots=V_{2\ell+1}(a^*)=0.$ By hypotheses,  for any $j\geq 3,$
\[
W_j^{n_j}(a)=\sum_{i=1}^\ell p_{2i+1,j}(a)V_{2i+1}(a),
\]
for some polynomials $ p_{2i+1,j}(a)$. Hence,
$W_j^{n_j}(a^*)=0$  for all $j\geq 3$, giving that $W_j(a^*)=0.$ As a
consequence, $f_{a^*}(x)=-x,$ a contradiction with our initial
assumption.

 Finally, the maximum cyclicity for any map $f_a,$ with
$a\ne\mathbf{0},$ is $\ell-1$ because of the first part of
Proposition \ref{p:lower}.
\end{proof}

We end this list of preliminary results with a remark about the
cyclicity  of  families of orientation preserving  diffeomorphisms.

\begin{nota}\label{P-cic0op}
For any family of maps $f_{a}(x)=x+\sum_{i\geq 2} c_{i}(a) x^i$ with
$a$ in an open set $\mathcal{V}\subseteq \R^{n}$, depending
continuously
 with respect to $a,$ the origin has
$2$\emph{-cyclicity} $0$. This holds because, given any $a=a^*$
there is a neighborhood of the origin and $a^*$ for which $f_a$ is
monotonous increasing.
\end{nota}

\section{Proof of Theorem \ref{T-cicl}}\label{S-prova}


For any $d\geq 2$ the family of maps \eqref{E-poly-fam-a} is a
$(d-1)$-parametric family, with parameters $a=(a_2,\ldots,a_{d})\in
\mathbb{R}^{d-1}$.  As we have already argued in the introduction,
if for some $a,$ $f_a$ has degree $k$ then $f_a\circ f_a$ has degree
$k^2.$ Hence the only involution in the family is the trivial one
$f_{\mathbf{0}}(x)=-x.$ To prove the second assertion of the
statement, notice that $x=0$ is a fixed point of $f_a$ and $f_a\circ
f_a$. Hence, the maximum number of global 2-periodic orbits of a
polynomial map in the family \eqref{E-poly-fam-a} is
$\left[(d^2-1)/2\right]$.

$(a)$  Consider first the case $d=2n,$ even.  We start proving that
its cyclicity is at least $d-2$. In this situation it is very easy
to prove that taking $a^*=(0,0,\ldots,0,1)$ the origin is a weak
fixed point of order $d-2=2(n-1),$ with
$W_{4n-1}(a^*)=V_{4n-1}(a^*)=-2n\ne0,$ because when
$f_{a^*}(x)=-x+x^{2n},$
\begin{align*}
f_{a^*}(f_{a^*}(x))=&x-x^{2n}+\big(-x+x^{2n}\big)^{2n}=x-x^{2n}+x^{2n}\big(1-x^{2n-1}\big)^{2n}\\
=&x-2n x^{4n-1}+O(x^{4n}).
\end{align*}
 To show that  the cyclicity of the map $f_{a^*}$
is $2(n-1)$ we will apply Proposition~\ref{p:lower2}. Therefore we
must prove that the vectors in $\mathcal{W}:=\{\nabla
W_{2k+1}(a^*),$ $k=1,2,\ldots, 2n-2\},$ are linearly independent,
where we recall that $\nabla=(\partial/\partial a_2,
\partial/\partial a_3,\ldots,
\partial/\partial a_{2n-1})$.

Using the quasi-degree properties  of the stability constants
$W_{2k+1}(a)$  proved in Proposition~\ref{L-quasi-homo},
 it is clear that for a general family of maps
\eqref{E-poly-fam-a} with $d=2n$, for any $W_{2k+1}(a),$
$k=1,2,\ldots, n-1,$ the only degree~1 monomial of each of them  is
$\alpha_{2k+1} a_{2k+1}$ for some real constants $\alpha_{2k+1}.$ To
determine these constants notice that when $f(x)=-x+ x^{2k+1}$ then
\[
f(f(x))=x-x^{2k+1}+\big(-x+x^{2k+1}\big)^{2k+1}=x-x^{2k+1}-x^{2k+1}\big(1+
O(x)\big)=x-2x^{2k+1}+O(x^{2k +2}).
\]
Hence, for these values of $k$, $\alpha_{2k+1}=-2.$ In consequence
\begin{equation}\label{nabla1}
\nabla W_{2k+1}(a^*)=(0,0,\ldots,-2,0,\ldots,0) ,\quad k=1,2,\ldots,
n-1,
\end{equation}
where the $-2$ is placed at the $2k$ position of the
$(2n-2)$-dimensional vector, because for these values of $k$ all the
other monomials of $W_{2k+1}(a)$ have degree at least 2, and their
derivatives, evaluated at $a^*$ vanish.

For $k$ from $n$ until $2n-2,$ and due again to the algebraic
property  given in Proposition~\ref{L-quasi-homo}, the corresponding
stability constant $W_{2k+1}(a)$  (again, for a general family of
maps \eqref{E-poly-fam-a} with $d=2n$)
 has no monomials of degree 1. Similarly
 it can have several monomials of degree 2, all
of them of the form $\beta_{s,t} a_{2s}a_{2t},$ for some real values
$\beta_{s,t},$ to be determined, and with $(s,t)\in\N^2,$ both
bigger than 1 and such that $s+t=k+1.$ Because we are only
interested on computing $\nabla W_{2k+1}(a^*),$ the only relevant
monomial of degree 2 in $W_{2k+1}(a)$ will be $\beta_{k+1-n,n}
a_{2(k+1-n)}a_{2n}.$ To obtain these values of $\beta_{k+1-n,n},$
consider $f(x)=-x+x^{2(k+1-n)}+x^{2n}.$ Similar computations than
the ones done above give that this coefficient  in $W_{2k+1}(a)$ is
$-2(k+1).$ Therefore $W_{2k+1}(a)$ has the monomial $-2(k+1)
a_{2(k+1-n)}a_{2n}$ and it holds that
\begin{equation}\label{nabla2}
\nabla W_{2k+1}(a^*)=(0,0,\ldots,-2(k+1),0,\ldots,0),\quad
k=n,n+1,\ldots, 2n-2,
\end{equation}
where the value $-2(k+1)$ is placed at the $2(k-n)+1$ position of
this $(2n-2)$-dimensional vector.

 Joining \eqref{nabla1} and \eqref{nabla2},  we
obtain that the vectors in $\mathcal{W}$ are linearly independent.
Hence we have proved that when $d=2n$ the cyclicity of the whole
family is at least $d-2,$  because for this specific value of
$a=a^*$ it is so.

Now we are going to consider the maps (\ref{E-poly-fam-a}) for small
values of $d.$

\noindent{\bf Case $\mathbf{ d=2}$.} In this simple case
$f_a(x)=-x+a_2\,x^2$ and $f_{a}(
f_a(x))=x-2\,a_2^2\,x^3+a_2^3\,x^4.$ The equation $f_{a}( f_a(x))=x$
only gives the solutions $x=0$ and $x=\frac{2}{a_2}$ which in fact
are fixed points of $f_a.$ Hence $f_a$ has not 2-periodic orbits.

\smallbreak

\noindent{\bf Case $\mathbf{ d=4}$.} In this case $f_a(
f_a(x))=x+V_3\,x^3+\sum_{j=4}^{16} W_j\,x^j$. It is straightforward,
either by hand, or using the Gr\"obner basis package in Maple that
we are under the hypotheses of Proposition \ref{p:upper} with $m=3.$
Hence the cyclicity of the family is  at most 2 and, therefore,
since we have proved that it is at least $d-2=2$, it is exactly 2.
As an example of the computations that we have done, next we give
some details of the first algebraic relations.

In this case \begin{align*}
 V_3=& -2{a_{{2}}}^{2}-2a_{{3}},\quad
  V_5= -6\,a_{{4}}a_{{2}}+4\,{a_{{3}}}^{2},\quad
  V_7=3\,a_{{2}}a_{{3}}a_{{4}}-4\,{a_{{
4}}}^{2},
\end{align*}
and it holds that $V_3=W_3,$
\begin{align*}
W_4=&-\frac{1}{2}\,a_2\,V_3, \quad
W_5=V_5+\frac12\,a_3\,V_3,\\
 W_6=&-\frac{3}{2}\,a_2\,V_5+\frac{1}{2}(a_4-a_2\,a_3)V_3,\\
 W_7=&V_7+\frac34(a_2^2-a_3)V_5-\frac14a_2a_4V_3,
\end{align*}
and $W_j\in\langle V_3,V_5,V_7\rangle,$ for $j=8,9,\ldots,16.$

\smallbreak

\noindent{\bf Case $\mathbf{ d=6}$.} As  when $d=4$, we want to
apply Proposition \ref{p:upper}. In this case we prove that we are
under the hypotheses of this result with $m=6,$ and hence the
cyclicity of the family will be at most $5.$ Indeed, using the
Maple's Gr\"obner basis package again we find that,
$$
W_j\in \langle V_3,V_5,V_7,V_9,V_{11},V_{13}\rangle \mbox{ for } 3\leq j\leq 36.
$$
 Moreover, it also holds that
$$
W_j^2\in \langle V_3,V_5,V_7,V_9,V_{11}\rangle :=\mathcal{I}
\mbox{ for } 3\leq j\leq 36,\mbox{ and also } W_{13}\not\in \mathcal{I},
$$
showing that we are under the hypotheses of Lemma~\ref{l:rad} with
$\ell=5,$ proving that the cyclicity of any map $f_a,$ with $a\neq
\mathbf{0},$ is at most  $\ell-1=4=d-2.$

Notice that the above two relations imply in particular that the
ideal $\mathcal I$ is not radical.

\smallbreak

\noindent{\bf Cases $\mathbf{ d=8,10}$.} Doing similar computations
that when $d=6$ we can apply the same results.

For $d=8$ we get that $m=10$ and $\ell=7$, because,
$$
W_j\in \langle V_3,V_5,\ldots,V_{19},V_{21}\rangle
\mbox{ for } 3\leq j\leq 64
$$
and no similar relation appears before. Moreover,
$$
W_j^2\in \langle V_3,V_5,\ldots,V_{13},V_{15}\rangle
\mbox{ for }  3\leq j\leq 64.
$$
Hence, by Proposition \ref{p:upper} the cyclicity of
$f_{\mathbf{0}}$ is at most 9 and, by Lemma \ref{l:rad} and the fact
that the cyclicity is at least $d-2=6$, we get the desired result.
We remark that for some $W_j$ there is no need to take $W_j^2$ to be
in the ideal, but it is essential for instance for $W_{17}.$

For $d=10,$  $m=14$ and $\ell=9$, we prove that the cyclicity of
$f_{\mathbf{0}}$ is at most 13 and that the cyclicity of
 any $f_a,$ for $a\neq \mathbf{0},$ is once again
$d-2=8.$ We remark that in this case it happens that
$$
W_j\in \langle V_3,V_5,\ldots,V_{27},V_{29}\rangle
\mbox{ for } 3\leq j\leq 100,
$$
without similar relations appearing before, and
$$
W_j^3\in \langle V_3,V_5,\ldots,V_{17},V_{19}\rangle:=\mathcal{I}
\mbox{ for } 3\leq j\leq 100.
$$
We remark that not all $W_j$ need the exponent 3 to be in
$\mathcal{I}.$ Nevertheless, for instance, neither $W_{21}$ nor
$W_{21}^2$ are in $\mathcal{I}.$

\smallbreak

$(b)$ When $d=2n+1$ is odd it is clear that the cyclicity of the
family is at least the one to the case of degree $2n$, that we have
proved that it is at least $2n-2.$ Hence it is at least $d-3.$

Now we are going to consider the cases $d=3,5,7,9.$

\medskip

\noindent{\bf Case $\mathbf{d=3}$.} Doing similar computations that
the ones of case $d=4$ we get that we are under the hypotheses of
Proposition \ref{p:upper} with $m=2.$ Hence an upper bound of the
cyclicity of this family is 1. To prove that this bound is attained
we take  $a^*=(a_{2}^*,a_{3}^*)=(1,-1).$ Then $V_3(a^*)=0$ and
$V_5(a^*)=4>0$. Since  $V_3(a)=-2{a_{{2}}}^{2}-2a_{{3}},$ it holds
that
\[
\nabla V_3(a^*)=\left.\frac{\partial}{\partial a_2}
V_3(a)\right|_{a^*=(1,-1)}= -4\ne0.
\]
Therefore the cyclicity of the map $f_{a^*}$ is 1, and so it is the
cyclicity of the family.

\noindent{\bf Case $\mathbf{ d=5}$.} Proceeding as in case $d=6,$
first we will get some upper bounds of the cyclicity. In fact we can
apply Lemma~\ref{l:rad} with $\ell=4,$ proving that the cyclicity of
any map $f_a,$ with $a\neq \mathbf{0},$ is at most $\ell-1=3=d-2,$ and
Proposition \ref{p:upper} with $m=5,$ showing that the cyclicity of
the family is at most $4.$ Next we present one example with cyclicity 3.

 By solving the system $\{V_3(a)=V_5(a)=V_7(a)=0\}$
with respect $a_2,\ldots,a_5$, and by direct inspection of its
 solutions, we obtain that taking
$a^*=(1,-1,(9+\sqrt{55})/2,-(23+3\sqrt{55})/2)$ it holds that
$$
V_3(a^*)=V_5(a^*)=V_7(a^*)=0,\, V_9(a^*)=1701+229\,\sqrt {55}>0.
$$
A computation shows that
$$
\det\left({\nabla} V_3(a^*), {\nabla V_5}(a^*)
{\nabla} V_7(a^*)\right)=\det
 \left(\begin {array}{ccc} -4&-27-3\,\sqrt {55}&-{\frac{1}{2}}\left(27-3\,
\sqrt {55}\right)\\ -2&-8&-136-18\,\sqrt {55}
\\ 0&-6&-39-4\,\sqrt {55}\end {array} \right)$$
$$
=5280+736\,\sqrt {55},
$$
where $\nabla=(\partial/\partial a_2,
\partial/\partial a_3,
\partial/\partial a_4)$. So, the three vectors ${\nabla} V_3(a^*), {\nabla V_5}(a^*)$ and
${\nabla} V_7(a^*),$ are linearly independent and therefore, by  Proposition~\ref{p:lower}, the
cyclicity of $f_{a^*}$ is exactly 3.

\smallbreak

\noindent{\bf Case $\mathbf{ d=7}$.} We start proving that the cyclicity is at least
 five by finding an example with this cyclicity.  Proceeding as in the above case,
or looking at the proof of item (iv), we find that taking
$a^*=(0,0,1,0,0,-2),$  we obtain that
$V_3(a^*)=V_5(a^*)=V_7(a^*)=V_9(a^*)=V_{11}(a^*)=0$ and $V_{13}=42$.
A computation gives
$$
\det\left(\nabla V_3(a^*),\nabla V_5(a^*),\nabla V_7(a^*),\nabla
V_9(a^*),\nabla V_{11}(a^*)\right)=\det \left( \begin {array}{ccccc}
0&-6&0&0&{\frac {11765}{121}}
\\ \noalign{\medskip}-2&0&0&-{\frac {716}{17}}&0\\ \noalign{\medskip}0
&0&-8&0&0\\ \noalign{\medskip}0&-2&0&0&-{\frac {515}{121}}
\\ \noalign{\medskip}0&0&0&-10&0\end {array} \right)$$
$$=-35200.
$$
 Hence the cyclicity of $f_{a^*}$ is 5, as desired.
Finally, using Maple again we get that
$$W_j\in \langle V_3,V_5,\ldots,V_{17}\rangle \mbox{ for } 3\le j\le 48$$ and
$$W_j^2\in \langle V_3,V_5,\ldots,V_{13}\rangle \mbox{ for } 3\le j\le 48.$$ Hence the
cyclicity of $f_{\mathbf{0}}$ is at most 7 and the cyclicity of
 $f_a,$ for any $a\neq \mathbf{0}$ is  at most
$d-2=5.$

\smallbreak

\noindent{\bf Case $\mathbf{ d=9}$.} Once again, some computations using the Maple's
Gr\"obner basis package give
$$W_j\in \langle V_3,V_5,\ldots,V_{23}\rangle \mbox{ for } 3\le j\le 81$$ and
$$W_j^3\in \langle V_3,V_5,\ldots,V_{17}\rangle \mbox{ for } 3\le j\le 81.$$ Hence
the cyclicity of $f_{\mathbf{0}}$ is at most 10 and the cyclicity of
$f_a,$  for any $a\neq \mathbf{0}$ is less or equal to $d-2=7.$ To
end the proof  we prove that there is a value of $a\neq \mathbf{0}$
such that the cyclicity at this value is~7.

 In this case,  the study of the solutions of the system of equations
described by the first five reduced stability constants is
complicated. So we will propose an alternative method for obtaining
weak fixed points of high order. This method is based on the
knowledge of the structure  of 1-dimensional involutions.

It is know that any analytic 1-dimensional non-trivial involution
$h$ can be written as
\[
h(x)=g(-g^{-1}(x)),
\]
where $g$ is an analytic diffeomorphism such that $g(0)=0,$ see
\cite{K}. Notice that it is straightforward to check that
$g(-g^{-1})$ is an involution.  Take any map of the form
\[
g(x)=x+\sum\limits_{j\geq 2} b_j x^j
\]
and compute the Taylor series of its inverse,
\[
g^{-1}(x)=x-b_2x^2+(2b_2^2-b_3)x^3+(-5b_2^3+5b_2b_3-b_4)x^4+\sum\limits_{j\geq
5} D_j(b) x^j,
\]
where $b=(b_2,b_3,\ldots)$ and we do not detail the polynomial
functions $D_j,$  that are given by the so called Bell polynomials.
It holds that $h\circ h=\operatorname{Id}.$

 Now, to find a map with a weak fixed point of high order, we can fix some
degree $d,$ and  consider the Taylor approximation of $h$ of
degree $d$, at the origin, $h_d=T_d(h)$. Then
\[
h_d(x)=-x+\sum_{j=2}^d B_j(b)x^j,
\]
where
\begin{equation}\label{e:Bs}
\begin{array}{l}
B_2(b)=2b_2,\,B_3(b)=-4b_2^2,\,B_4(b)=10b_2^3-4b_2b_3+2b_4,\\
B_5(b)=-28b_2^4+24b_2^2b_3-12b_2b_4,
\end{array}
\end{equation}
and $B_j(b),$ for $j=6,\ldots,d$ are some polynomials that we do not detail. This map has a
high order weak fixed point at the origin. For instance when $d=9,$
it holds that
\begin{equation*}
h_9(h_9(x))=x+W_{11}(b)x^{11}+\sum_{j=12}^{81} W_j(b)x^j.
\end{equation*}
Now, to increase the level of weakness of the fixed point, that is the order of $h_9$, we have to
select the values of $b$ such that the associated stability constants up to order 15 vanish, i.e.
\begin{equation}
\begin{cases}W_{11}(b)=W_{13}(b)=W_{15}(b)=0,\end{cases}
\end{equation}
where we omit the expression of these stability constants. Since
$W_{11}(b)$ linear with respect $b_7$ we can isolate and fix this parameter, obtaining
$$\begin{array}{rl}
b_7:=&\dfrac {1}{4b_2 \left( 107\,b_2^{3}+6\,b_2b_3-3
\,b_4 \right)} \,\left(20774b_2^{10}-64272b_2^{8}b_3+32136b_2^{7}
b_4+52962b_2^{6}b_3^{2}\right.\\
&-7644b_2^{6}b_5-
41496b_2^{5}b_3b_4-9464b_2^{4}b_3^{3}+
3822b_2^{5}b_6+4836b_2^{4}b_3b_5+6552{b_
{{2}}}^{4}b_4^{2}\\
&+6942b_2^{3}b_3^{2}b_4-507{
b_2}^{2}b_3^{4}-1776b_2^{3}b_3b_6-1348b_2^{3}b_4b_5+300b_2^{2}b_3^{2}b_5-684b_2^{2}b_3b_4^{2}\\
&+564b_2b_3^{3}b_4+214
b_2^{3}b_8+246b_2^{2}b_4b_6-12b_2^{2
}b_5^{2}-114b_2b_3^{2}b_6-204b_2b_3b_{
{4}}b_5-50b_2b_4^{3}\\
&\left.-156b_3^{2}b_4^{2}+
12b_2b_3b_8+12b_2b_5b_6+60b_3b_4
b_6+24b_4^{2}b_5-6b_4b_8-3b_6^{2}
\right).
\end{array}
$$
To reduce the number of parameters we impose $b_2=1$ and $b_3=0$, and solve the system
$\{W_{13}(b)=W_{15}(b)=0\}$ using the Maple algebra software, obtaining the following solution,
among others:
$b_4=\xi,$ $b_5$ as a free parameter and $b_{{6}}=n(b_5,\xi)/d(b_5,\xi)$
where
$$\begin{array}{rl}
n(b_5,\xi)=&-4830249480\xi^{9}+78255450\xi^{8}b_5-309996323910\xi^
{8}+121885399860\xi^{7}b_5\\
&+499588480916\xi^{7}-
3569620983180\xi^{6}b_5+433844538538740\xi^{6}\\
&-
3036308656220\xi^{5}b_5+10120115599755700\xi^{5}-
1400107036991768\xi^{4}b_5\\
&+78554454691772584\xi^{4}-
16364417170088484\xi^{3}b_5+278979787186921660\xi^{3}\\
&-
60913553653703380\xi^{2}b_5+434487144164761772\xi^{2}-
150424031357777588\xi b_5\\
&-1476344096012712444\xi+
253882004776386078b_5-1551280344412627458,\\
d(b_5,\xi)=&39127725{\xi}^{8}+60942699930{\xi}^{7}-1784810491590{\xi}^{6}-
1518154328110{\xi}^{5}\\
&-700053518495884{\xi}^{4}-8182208585044242
{\xi}^{3}-30456776826851690{\xi}^{2}\\
&-75212015678888794\xi+
126941002388193039,
\end{array}
$$
and where $\xi$ is any real root of the polynomial
$$\begin{array}{rl}
P(x)=&160228033875{x}^{16}+221432009870400{x}^{15}+13936473199884004{x
}^{14}\\
&-683923454204391464{x}^{13}+2642995488208403832{x}^{12}\\
&-
385227003687957189136{x}^{11}-3012116857431809290604{x}^{10}\\
&+
45026084431427989413608{x}^{9}+752080887518088204729142{x}^{8}\\
&+
5032896522827017198516064{x}^{7}+17779108732214526516315308{x}^{6}\\
&+29817171191523879926181416{x}^{5}-14212793325606052484090592{x}^{
4}\\
&-123365732211297823524968592{x}^{3}-274115367296634168846158244{
x}^{2}\\
&-325682563327763246441199080x-133940574254498343421555617.
\end{array}$$
Notice that, using the Sturm algorithm, one can check that $P(x)$
has 8 different simple real roots.

Finally we set $b_5=0$. With this choice of the parameters each constant $W_j(b)$ writes as
$W_j(\xi)$. A computation shows that for $j=11,\ldots,16$:
$$\begin{array}{l}
\mathrm{Resultant}\left(P(\xi),\mathrm{Numer}\left(W_j(\xi)\right);\xi\right)=0,\\
\mathrm{Resultant}\left(P(\xi),\mathrm{Denom}\left(W_j(\xi)\right);\xi\right)\neq0,
\end{array}
$$
and
$$
\begin{array}{l}
\mathrm{Resultant}\left(P(\xi),\mathrm{Numer}\left(W_{17}(\xi)\right);\xi\right)\neq 0,\\
\mathrm{Resultant}\left(P(\xi),\mathrm{Denom}\left(W_{17}(\xi)\right);\xi\right)\neq 0.
\end{array}
$$
 Hence, when $x=\xi^*$ is any of the real roots of
$P(x)$ the map $h_9$ has a weak fixed point of order $7$. Now we
prove that it has cyclicity 7. Using the expressions of the
 functions $B_j$ (see \eqref{e:Bs})  we set
$a_j^*=B_j(\xi^*)$ for $j=2,\ldots 9$ and take
$a^*=(a_2^*,\ldots,a_9^*)$. A computation gives that
$$
\det\left(\nabla V_3(a^*),\ldots,\nabla V_{15}(a^*)\right)=\dfrac{R(\xi)}{Q(\xi)},
$$
where $R$ and $Q$ are co-prime polynomials with degree 77 and 68 respectively in $\xi$.
Again, one can check that
$\mathrm{Resultant}\left(P(\xi),R(\xi);\xi\right)\neq 0,$ and
$\mathrm{Resultant}\left(P(\xi),Q(\xi);\xi\right)\neq 0,$
hence the vectors $\nabla V_3(a^*),\ldots,\nabla V_{15}(a^*)$ are
linearly independent and, by Proposition \ref{p:lower}, the
cyclicity of $f_{a^*}$ is at least $7$. This ends the proof of
statements (b) (i)--(iii).


\medskip

 To prove statement (iv), we consider for $d=4m+3$:
$$f_{a^*}(x)=-x+x^{2m+2}-(m+1)\,x^{4m+3}.$$
A routine computation shows that
$$f_{a^*}(f_{a^*}(x))=x+\frac{(m+1)(5m+4)(4m+3)}{3}\,x^{8m+5}+O(x^{8m+6}).$$
 Hence $f_{a^*}$ has a weak fixed point of order $d-2=4m+1$ as we wanted to
 show.

\section{Poincar\'e maps and $2$-cyclicity}\label{SS-bifurcation}

Locally orientation reversing diffeomorfisms appear naturally when studying
\emph{the Poincar\'e maps} associated to the origin of planar differential systems  of the form
\begin{equation}\label{E-vf}
\begin{cases}
\begin{array}{l}
\dot{x}=-y+P(x,y),\\
\dot{y}=\,\,x+Q(x,y),
\end{array}
\end{cases}
\end{equation}
where $P$ and $Q$ are analytic functions starting with at least second order terms.
It is well known that the origin of the above system is \emph{monodromic}, i.e. there is a well defined
associated Poincar\'e map. In this situation, using polar coordinates $r^2=x^2+y^2$
and $\theta=\arctan(y/x)$ the solution of \eqref{E-vf} that passes through the point
$(x,0)$ with $x>0$ small enough can be expressed by
$r(\theta;x)=x+\sum_{i\geq 2} a_i(\theta) x^i$, and the Poincar\'e map is given by $\Pi(x)=r(2\pi;x)$.

Let $\Pi_+(x)$ be the map defined over an interval $(0,\epsilon)\subset \R^+$ given by
$\Pi_+:(0,\epsilon)\rightarrow \R^-$ where $(0,\epsilon)$ is on the semi-axis $OX^+$,
such that it gives the first intersection, in positive time, of the orbit that at time
$t=0$ passes through the point $(x,0)$. We call this map the \emph{half-return map}.
 In \cite{CGP} it is proved that
$\Pi_+(x)=-r(\pi;x)$, hence it is of the form \eqref{E-maps}. As can
be seen in this reference, $\Pi_+(x)$ has an analytic extension to
$\mathbb{R}$.
 Using this analytic extension,
the authors prove  that $\Pi=\Pi_+\circ \Pi_+$.

 It is clear, then, that given a parametric family of vector fields of the form
$$
X_a(x,y)=\left(-y+P_a(x,y)\right)\frac{\partial}{\partial x}+
\left(x+Q_a(x,y)\right)\frac{\partial}{\partial y},
$$ with $a\in\R^n$ and $P_a(x,y)$ and $Q_a(x,y)$ starting with second order terms,
the cyclicity of $X_a$ (that is, the number of small amplitude limit
cycles of the differential equation associated to $X_a$) is exactly
the cyclicity of the associated family of maps $\Pi_{+,a}(x)$ (the
number of $2$-periodic orbits). Conversely, observe that
 the following result proves that any given map of the form \eqref{E-maps}
 is conjugate with the corresponding half-return map of a polynomial vector field.

\begin{propo}\label{P-vf}
Given an analytic map with $f(0)=0$ and $f'(0)=-1$, there exists a polynomial
vector field of the form \eqref{E-vf} such that $f(x)$ is locally
$\mathcal{C}^\infty$-conjugate to the half-return map $\Pi_+(x)$ of the vector field.
\end{propo}

\begin{proof}  Suppose that $f$ is an involution. By the Bochner
linearization Theorem \cite{MZ},  the local diffeomorphism
  $\psi(x)= x-f(x)$ conjugates $f$ with the linear map $L(x)=-x$ (it is
  straightforward to check that $\psi\circ f=L\circ \psi$). Hence, $f$ is
  analytically conjugate with the  half-return map of a the linear center
$$\begin{cases}
\begin{array}{l}
\dot{x}=-y,\\
\dot{y}=x.
\end{array}
\end{cases}$$

Suppose now, that $f$ is not an involution. Following \cite{T},
there exists a local $\mathcal{C}^\infty$-diffeomorphism
$\varphi_1$, that conjugates $f$ with its
$\mathcal{C}^\infty$-normal form
$$
f_{N}(x)=-x+\sigma x^{2\ell+1}+c x^{4\ell+1},
$$
where $\sigma=\pm 1$.

Consider the polynomial vector field given by
\begin{equation}\label{E-vf-2}
\begin{cases}
\begin{array}{l}
\dot{x}=-y+x\left(\delta (x^2+y^2)^{2\ell}+\gamma (x^2+y^2)^{4\ell}\right),\\
\dot{y}=x+y\left(\delta (x^2+y^2)^{2\ell}+\gamma (x^2+y^2)^{4\ell}\right),
\end{array}
\end{cases}
\end{equation} with $\delta=-\sigma/\pi$ and $\gamma=-(c+(2\ell+1)\sigma^2/2)/\pi$.
We claim that, using the notation introduced above,
\begin{equation}\label{E-pi-mes-r}
\Pi_+(x)=-r(\pi;x)=-x+\sigma x^{2\ell+1}+c
x^{4\ell+1}+O(x^{4\ell+2}),
\end{equation}
and therefore, there exists a $\mathcal{C}^\infty$-diffeomorphism $\varphi_2$,
that conjugates $\Pi_+$ with $f_{N}$. In consequence,
$$
f_N=\varphi_1^{-1}\circ f \circ\varphi_1\,\mbox{ and }\,
f_N=\varphi_2^{-1}\circ \Pi_+ \circ \varphi_2,
$$
so
$$
f=(\varphi_1\circ\varphi_2^{-1})\circ \Pi_+ \circ (\varphi_2\circ
\varphi_1^{-1})
$$
and $f$ is conjugate with $\Pi_+$. To finish, we prove \eqref{E-pi-mes-r}. We apply
similar arguments than the ones used in the proof of Lemma 2.7 in
\cite{T}.

\medskip

Observe that the system \eqref{E-vf-2} has the associated polar equation
\begin{equation}\label{E-polar}
\dot{r}=\delta r^{2\ell+1}+\gamma r^{4\ell+1},
\end{equation}
with analytic solution
$
r(\theta;x)=\sum_{i\geq 1} a_i(\theta) x^i.
$
By substituting this expression in \eqref{E-polar}, taking into account that $r(0;x)=x$, and
comparing powers we obtain that $a_i'(\theta)=0$ for all $i=1,\ldots,2\ell$, so
$a_1(\theta)\equiv 1$ and $a_i(\theta)\equiv 0$ for all $i=2,\ldots,2\ell$. Applying the
same argument we have
$$\begin{array}{rl}
\sum_{i\geq 2\ell+1}a_i'(\theta)x^i&=\delta\left(x+\sum_{i\geq 2\ell+1}a_i(\theta)x^i\right)^{2\ell+1}+
\gamma\left(x+\sum_{i\geq 2\ell+1}a_i(\theta)x^i\right)^{4\ell+1}\\
&\\
&=\delta
x^{2\ell+1}+\left(\delta(2\ell+1)a_{2\ell+1}(\theta)+\gamma\right)x^{4\ell+1}+O(x^{4\ell+2}).
\end{array}
$$
Hence, comparing powers, integrating term by term, and using again
that $r(0;x)=x$ we have that $a_{2\ell+1}(\theta)=\delta \, \theta$,
$a_i(\theta)\equiv 0$ for all $i=2\ell+2,\ldots,4\ell$, and
$a_{4\ell+1}(\theta)=\gamma \,\theta+\delta^2(2\ell+1)\theta^2/2$.
The result follows using that $\delta=-\sigma/\pi$ and
$\gamma=-(c+(2\ell+1)\sigma^2/2)/\pi$.~\end{proof}

 The result above establishes that each map \eqref{E-maps} is conjugate to the
 half-return map of a polynomial vector field,
but we remark that given a map  \eqref{E-maps} it is not easy to prove that \emph{it is} the
 half-return map of a polynomial vector field.

\medskip

We end the paper showing that there exist families of type
(\ref{E-mapsaestrella}) with a single parameter having cyclicity $k$
for any $k\in\N.$ This is a consequence of the results in \cite{GG}.
Indeed, in this reference it is shown that for any $k\in\N$ there
exists a suitable choice of fixed values of
$\alpha_1,\ldots,\alpha_k\in \R$, such that the one-parametric
family of vector fields with associated differential system
$$
\begin{cases}
\begin{array}{l}
\dot{x}=-y+x(x^2+y^2)\left(a^k+\alpha_1 a^{k-1} r^2+\cdots+\alpha_{k-1} a r^{2(k-1)}+\alpha_k r^{2k}\right),\\
\dot{y}=x+y(x^2+y^2)\left(a^k+\alpha_1 a^{k-1} r^2+\cdots+\alpha_{k-1} a r^{2(k-1)}+\alpha_k r^{2k}\right),
\end{array}
\end{cases}
$$
with $r^2=x^2+y^2$, has cyclicity $k$ and, in consequence
the one-parametric family of locally orientation reversing analytic
diffeomeorphisms $\Pi_{+,a}$ also has cyclicity $k$.

\end{document}